\RequirePackage{amsmath}% to remove a bug about redefining \vec
\RequirePackage{fix-cm}
\documentclass[smallextended]{svjour3}       % onecolumn (second format)
\smartqed  % flush right qed marks, e.g. at end of proof
\usepackage{graphicx}
\usepackage{amsmath}
\usepackage{amssymb}
\journalname{Graphs and Combinatorics}

\begin{document}

\title{On the number of Hamiltonian cycles and polynomial invariants of graphs%\thanks{Grants or other notes
}


\author{Yi Bo
}


\institute{Yi Bo \at
              Department of Mathematics  \\Harbin Institute of Technology \\ Harbin 150001, P.R.CHINA \\
              \email{1427341332@qq.com}           %  \\
}

\date{Received: date / Accepted: date}

\maketitle

\begin{abstract}
Firstly, for a general graph, we find a recursion formula on the number of Hamiltonian cycles and one on cycles. By this result, we give some new polynomial invariants. Secondly, we give a condition to tell whether a polynomial defined by recursion on edges is a invariant. Then, we give an generaliztion of the Tutte polynomial. Finally, We have a try on distinguishing different graphs by using these polynomials.
 %\subclass{Primary 05C31 \and 05C45; Secondary 05C30} %if use MSC(2010)
\end{abstract}

\section{Introduction}
            The Hamiltonian cycle problem is a famous open problem in graph theory. There are two topics to study this problem, existence and enumeration. Existence is to say that, if a graph have some properties, there must be a Hamiltonian cycle in it. Enumeration is to say (about) how many Hamiltonian cycles are there in a graph with some properties. To know whether there is a Hamiltonian cycle in a graph is NP-complete and to know how many is \#P-complete, so it's difficult to study this question for a general graph. We will study enumeration of Hamiltonian cycles in this article.

            Polynomial invariants is a good tool to study a mathematical object (especially in knot theory and graph theory), because there are many invariants contained in the polynomials. For example, the chromatic polynomial which is defined by George David Birkhoff in \cite{P} to attack the four color theorem is very useful.

            The Tutte polynomial is also one of the most important and beautiful polynomial invariants of graphs. Its properties and relationships with other polynomial invariants can be found in \cite{GP,GT,dxs}.

            The following is a simple introduction of notation in graph theory (cf.\cite{EC}) to be used in this article. More details about graph theory is in \cite{GT}.
    \subsection{Preliminaries}
            A finite graph is a triple $G=(V,E,\varphi{})$, where $V$ is a finite set of vertices and denoted by $V(G)$ (sometimes short for $V$), $E$ is a finite set of edges and denoted by $E(G)$, and function $\varphi{}:$ $E\rightarrow{}V\ \&\ V$, where $V\ \&\ V$ is the set of unordered pair of $V$. If both $V$ and $E$ are empty sets, $G$ is an empty graph and denoted by $\phi{}$. We suppose a graph is nonempty unless we state it explicitly. If $\varphi{}(e)=\{u,v\}$, we say that $u$ and $v$ are adjacent, and that $u$ and $e$, $v$ and $e$ are incident. If $u=v$, we call $e$ is a loop. We denote a graph with only one vertex and $n$ loops by $K_{1}^{n}$. The degree of $v$ is the number of edges incident with $v$ (loops count twice), which is denoted by $deg(v)$. We denote $e$ by $uv$, although $\varphi{}$ may not be injective, when we only focus on $e$.

            A sequence $v_{0}e_{1}v_{1}e_{2}...e_{n}v_{n}$, where $n\geq{}0$, $v_{i}\in{}V$ and $e_{i}\in{}E$, is called a $v_{0}v_{n}$-walk of $G$, if any two consecutive terms are incident. A walk is a path if all the vertices are different; a cycle if $n\geq{}1$ and all the vertices and edges are different except for $v_{0}=v_{n}$. We denote the number of cycles of $G$ by $c(G)$. We call $u$ and $v$ are connected if there is a $uv$-walk in $G$. An equivalence class of connected vertices is called a connected component. We denote the number of connected components of $G$ by $k(G)$. A graph $G$ is connected, if $k(G)=1$. By $G=G_{1}+G_{2}$ we denote $G$ is the disjoint union of $G_{1}$ and $G_{2}$, i.e., $k(G)\geq{}2$ (because we suppose they are nonempty without statement).

            A graph $H=(V_{2},E_{2},\varphi{}_{2})$ is called a subgraph of $G=(V_{1},E_{1},\varphi{}_{1})$ ,which is denoted by $H\subseteq{}G$, if $V_{2}\subseteq{}V_{1}$, $E_{2}\subseteq{}E_{1}$ and $\varphi{}_{1}|_{E_{2}}=\varphi{}_{2}$. If $V_{2}=V_{1}$, we call $H$ is a spanning subgraph of $G$, denoted by $H\subseteq\subseteq{}G$. A cycle is called a Hamiltonian cycle, if it's a spanning subgraph. We denote the number of Hamiltonian cycles of $G$ by $h(G)$. A subgraph of $G$ is called a $j$-matching, if it have $2j$ vertices and every vertex have degree $1$. We denote the number of $j$-matching of $G$ by $m_{j}(G)$.

            For an edge $uv$ of $G$, we denote edge deletion of $uv$ by $G-uv$; edge contraction of $uv$ by $G\ /\ uv$, i.e., merging $u$ and $v$ by deleting (contracting) $uv$; vertex deletion of $u$ by $G-u$ (also deleting the edges incident with $u$), and similarly we have $G-v$ and $G-u-v$. We suppose $uv\in{}E$ without statement, when we use this operations.

            We denote the coefficient of $x^{k}$ in $f(x)$ by $[x^{k}]f(x)$, and the number of elements of set $A$ by $\#A$.
    \subsection{Outline}
            The following are the outline of the article.

            Section 2: Theorem 1 give a recursion formula on the number of Hamiltonian cycles, and by this result we define a polynomial. Lemma 2 give a recursion formula on the number of cycles, and by this result we define cycle polynomial.

            Section 3: Definition 6 merge the cycle polynomial and the matching polynomial. Theorem 4 give a sufficient and necessary condition to tell whether a polynomial defined by recursion on edges is an invariant of graphs.

            Section 4: Definition 8 give an generaliztion of the Tutte polynomial, and Proposition 1 give some of its properties.

            Section 5: We have some discussion on distinguishing different graphs by using polynomials.

            In this paper, Theorem 1 and Theorem 4 are the main theorems.
\section{On Hamiltonian cycles and cycles}
    \subsection{Hamiltonian cycles}
            We have a beautiful recursion formula on the number of Hamiltonian cycles.
        \begin{theorem}
            For a graph $G:$

            $(1)h(G)=h(G-uv)+h(G\ /\ uv)-h(G-u)-h(G-v)$, $u\neq{}v;$

            $(2)h(G_{1}+G_{2})=0;$

            $(3)h(K_{1}^{n})=n.$
        \end{theorem}
            \begin{proof}
                For $uv\in{}E$, $u\neq{}v$, $h(G-uv)$, obviously, equals the number of Hamiltonian cycles which don't include the edge $uv$. And for the pre-image (of the operation, the edge contraction of $uv$) of a cycle in $G\ /\ uv$ (that is to say, for a cycle in $G\ /\ uv$, you can extend the $uv$ back again to see from $G$ how we get the cycle in $G\ /\ uv$ by the edge contraction of $uv$), there are $3$ situations:

                $S_{1}$: there is $1$ edge incident with $u$ and $v$ respectively;

                $S_{2}$: there are $2$ edges incident with $u$;

                $S_{3}$: there are $2$ edges incident with $v$.

                The number of (cycles in $G\ /\ uv$ in) $S_{1}$ (short for $\#S_{1}$) equals the number of Hamiltonian cycles which include the edge $uv$. And in fact, $\#S_{2}=h(G-v)$, i.e., the number of cycles which include all the vertices of $G$ expect $v$. And similarly we have $\#S_{3}=h(G-u)$. And obviously, $\#S_{1}+\#S_{2}+\#S_{3}=h(G\ /\ uv)$. So the number of Hamiltonian cycles which include the edge $uv$ is $h(G\ /\ uv)-h(G-u)-h(G-v)$.

                So we have $h(G)=h(G-uv)+h(G\ /\ uv)-h(G-u)-h(G-v)$, $u\neq{}v$.

                Obviously for $(2)$ and $(3)$, and by using $(1)$ can make the size of a graph down, so we can get $h(G)$ by the recursion.
            \end{proof}

            It's so beautiful that we think someone may have got and proved it. But we didn't find any similar result. If you know any similar result, please give us a email. Thank you very much.

            In the proof above, thinking the pre-image of edge contraction is a key point, and so also will be in the following.
        \begin{definition}
            We call $H$ a $k$-components Hamiltonian cycle(or we say $2$-factor with $k$-components) of $G$, if $H\subseteq\subseteq{}G$, $k(H)=k$, and $\forall{}v\in{}V(H)$, $deg(v)=2$. We denote the number of $k$-components Hamiltonian cycles of $G$ by $h_{k}(G)$.
        \end{definition}

            Be careful about the ``cycle'' above. We just so call it, although it's not really a cycle when $k\geq{}2$. Similar to the proof of Theorem 1, we have:
        \begin{lemma}
            For a graph $G:$

            $(1)h_{k}(G)=h_{k}(G-uv)+h_{k}(G\ /\ uv)-h_{k}(G-u)-h_{k}(G-v)$, $u\neq{}v;$

            $(2)h_{k}(G_{1}+G_{2})=\sum_{i}h_{i}(G_{1})h_{k-i}(G_{2});$

            $(3)h_{1}(K_{1}^{n})=n$; $h_{k}(K_{1}^{n})=0$, $k\neq{}1.$
        \end{lemma}

            Then let's rewrite Lemma 1 in a beautiful way.
        \begin{definition}
            For a graph $G$, $H(G,s)$ (sometimes short for $H(G)$) is called the Hamiltonian cycle polynomial, which is got from:

            $(1)H(G)=H(G-uv)+H(G\ /\ uv)-H(G-u)-H(G-v)$, $u\neq{}v;$

            $(2)H(G_{1}+G_{2})=H(G_{1})H(G_{2});$

            $(3)H(K_{1}^{n})=ns.$
        \end{definition}

            Because of Lemma 1, Definition 2 is well-defined, and we have:
        \begin{theorem}
            For a graph $G$, $[s^{k}]H(G,s)=h_{k}(G).$
        \end{theorem}

            That is to say, $H(G)$ is the generating function of Hamiltonian cycles.

            Although we can define $H(G,s)$ at first, we didn't. Because we want to think polynomial defined by recursion instead of generating function. We get a generating function by chance.
    \subsection{Cycles}
            Similar to Hamiltonian cycles, we will do similar things to cycles.
        \begin{definition}
            We call $H$ a $k$-components $l$-length cycle of $G$, if $H\subseteq{}G$, $k(H)=k$, $\#E(H)=l$, and $\forall{}v\in{}V(H),deg(v)=2$. We denote the number of $k$-components $l$-length cycle of $G$ by $c_{k,l}(G)$.
        \end{definition}

            Similarly, it's not really a cycle when $k\geq{}2$.
        \begin{lemma}
            For a graph $G:$

            $(1)c_{k,l}(G)=c_{k,l}(G-uv)+c_{k,l-1}(G\ /\ uv)-c_{k,l-1}(G-u)-c_{k,l-1}(G-v)+c_{k,l-1}(G-u-v)$, $u\neq{}v;$

            $(2)c_{k,l}(G_{1}+G_{2})=\sum_{i,j}c_{i,j}(G_{1})c_{k-i,l-j}(G_{2});$

            $(3)c_{1,1}(K_{1}^{n})=n;$ $c_{0,0}(K_{1}^{n})=1;$ $c_{k,l}(K_{1}^{n})=0,$ $(k,l)\neq{}(0,0)$ or $(1,1);$ $c_{0,0}(\phi{})=1;$ $c_{k,l}(\phi{})=0$, $(k,l)\neq{}(0,0).$
        \end{lemma}
            \begin{proof}
                For cycles in $G$, and an edge $uv\in{}E$, $u\neq{}v$, we classify cycles by whether they contain $u$, $v$ or $uv$. There are $5$ kind of cycles, whose number is denoted by$:$

                $c^{(1)}:$ contain $u$, $v$, $uv$;

                $c^{(2)}:$ contain $u$, $v$, not contain $uv$;

                $c^{(3)}:$ contain $u$, not contain $v$, $uv$;

                $c^{(4)}:$ contain $v$, not contain $u$, $uv$;

                $c^{(5)}:$ not contain $u$, $v$, $uv$;\\
                We also use subscript to represent its number of components and length. We have:

                $c_{k,l}(G)=c^{(1)}_{k,l}+c^{(2)}_{k,l}+c^{(3)}_{k,l}+c^{(4)}_{k,l}+c^{(5)}_{k,l},$\\
                and(In the case of $G\ /\ uv$, we will use the method in Theorem 1, i.e., pre-image.):

                $c_{k,l}(G-uv)=c^{(2)}_{k,l}+c^{(3)}_{k,l}+c^{(4)}_{k,l}+c^{(5)}_{k,l};$

                $c_{k,l-1}(G\ /\ uv)=c^{(1)}_{k,l}+c^{(3)}_{k,l-1}+c^{(4)}_{k,l-1}+c^{(5)}_{k,l-1};$

                $c_{k,l-1}(G-u)=c^{(4)}_{k,l-1}+c^{(5)}_{k,l-1};$

                $c_{k,l-1}(G-v)=c^{(3)}_{k,l-1}+c^{(5)}_{k,l-1};$

                $c_{k,l-1}(G-u-v)=c^{(5)}_{k,l-1}.$\\
                So we have $c_{k,l}(G)=c_{k,l}(G-uv)+c_{k,l-1}(G\ /\ uv)-c_{k,l-1}(G-u)-c_{k,l-1}(G-v)+c_{k,l-1}(G-u-v)$, $u\neq{}v$.

                Obviously $(2)$ and $(3)$ are right, and by using $(1)$ we can make the size of a graph down, so we can get $c_{k,l}(G)$ by the recursion.
            \end{proof}

            Similarly, Lemma 2 can be rewritten in a beautiful way.
        \begin{definition}
            For a graph $G$, $C(G,s,t)$ (sometimes short for $C(G)$) is called the cycle polynomial, which is got from:

            $(1)C(G)=C(G-uv)+tC(G\ /\ uv)-tC(G-u)-tC(G-v)+tC(G-u-v)$, $u\neq{}v;$

            $(2)C(G_{1}+G_{2})=C(G_{1})C(G_{2});$

            $(3)C(K_{1}^{n})=1+nst$; $C(\phi{})=1.$
        \end{definition}

            Because of Lemma 2, Definition 4 is well-defined, and we have:
        \begin{theorem}
            For a graph $G$, $[s^{k}t^{l}]C(G,s,t)=c_{k,l}(G).$
        \end{theorem}

            Similarly, $C(G)$ is the generating function of cycles.
\section{Well-defined polynomial invariants}
            The generating function of matchings is well-known. There are many books about it. We just review it here for the following works.
        \begin{definition}
            For a graph $G$, $M(G,r)$ (sometimes short for $M(G)$) is called the matching polynomial, which is got from:

            $(1)M(G)=M(G-uv)+rM(G-u-v)$, $u\neq{}v;$

            $(2)M(G_{1}+G_{2})=M(G_{1})M(G_{2});$

            $(3)M(K_{1}^{n})=1$; $M(\phi{})=1.$
        \end{definition}

            $M(G)$ is well-defined, and: For a graph $G$, $[r^{j}]M(G,r)=m_{j}(G).$

            Then, let's merge $C(G)$ and $M(G)$:
        \begin{definition}
            For a graph $G$, $D(G,r,s,t)$ (sometimes short for $D(G)$) is got from:

            $(1)D(G)=D(G-uv)+tD(G\ /\ uv)-tD(G-u)-tD(G-v)+rD(G-u-v)$, $u\neq{}v;$

            $(2)D(G_{1}+G_{2})=D(G_{1})D(G_{2});$

            $(3)D(K_{1}^{n})=1+nst$; $D(\phi{})=1.$
        \end{definition}

            Be careful about this definition, not like the polynomials defined above (which are generating functions), it may not be well-defined (really an invariant). So we have to prove that it's well-defined. For getting a universal result, we have to make some preparation.
        \begin{definition}
            We say a polynomial invariant $F(\cdot)$ is trivial, if for any graph $G$, $F(G)=(F(K_{1}))^{\#V(G)}$, i.e., $F(\cdot)$ only have the information of how many vertices are there in a graph.
        \end{definition}

        Now we give a simple lemma.

        \begin{lemma}
            For a polynomial invariant $F(\cdot)$ with the condition $F(G_{1}+G_{2})=F(G_{1})F(G_{2})$, if for any graph $G$ and arbitrary different vertices $u$, $v$, $w$ of $G$, $F(G-u-v)=F(G-v-w)$, then $F(\cdot)$ is trivial.
        \end{lemma}
            \begin{proof}
                For any graph $G$, we make a graph $H$ from $G$ by the following way:

                $(1)$Choose an arbitrary vertex of $G$, and name it by $u$;

                $(2)$$V(H)=V(G)\cup{}\{v,w\}$, where $v$, $w\notin{}V(G)$;

                $(3)$$E(H)=E(G)\cup{}\{uv,vw\}$.

                Now we get $F(H-u-v)=F(H-v-w)$. That is to say $F(G-u+w)=F(G)$. Because $w\notin{}V(G)$, $F(G-u+w)=F(G-u)F(K_{1})$. That is to say for any $G$ and arbitrary vertices $u$ of it, $F(G)=F(G-u)F(K_{1})$. By mathematical induction, we have $F(G)=(F(K_{1}))^{\#V(G)}$, i.e., $F(\cdot)$ is trivial.
            \end{proof}

        Now the preparation is over, let's give a universal result.
        \begin{theorem}
            For a graph $G$, the polynomial got from:

            $(1)F(G)=aF(G-uv)+bF(G\ /\ uv)+cF(G-u)+cF(G-v)+dF(G-u-v)$, $u\neq{}v$;

            $(2)F(G_{1}+G_{2})=F(G_{1})F(G_{2});$

            $(3)F(K_{1}^{n})=f(n)$; $F(\phi{})=1$\\
            is well-defined if and only if $bc+cc+d-ad=0$ or $F(\cdot)$ is trivial.
        \end{theorem}
            \begin{proof}
                In this proof, we use $u$, $v$, $w$, $x$ as different vertices in $G$. And use $R_{uv}$ to present that we use $(1)$ on the edge $uv$, i.e., $R_{uv}F(G)=aF(G-uv)+bF(G\ /\ uv)+cF(G-u)+cF(G-v)+dF(G-u-v)$, $uv\in{}E$. We suppose $R_{uv}F(G)=F(G)$, if $uv\notin{}E$(only in this proof).

                By using $R_{uv}$, $uv$ is not contained in the graphs we get, so
                \begin{eqnarray}
                \Bigg(\prod_{u\neq{}v}R_{uv}\Bigg)F(G)
                \end{eqnarray}
                contains no edges which are not loops, i.e., only contains loops. For proofing that the polynomial is well-defined, we should show the following $3$ points:

                $P_{1}:$ `Formula (1)' doesn't change by the order of $R_{uv}$, i.e., is commutative;

                $P_{2}:$ the polynomial doesn't change by the order of $(1)$ and $(2)$;

                $P_{3}:$ the polynomial of any graphs that only contains loops is uniquely determined.

                In fact, $P_{3}$ is obvious, because of the commutative law of multiplication. And for $P_{2}$, also obvious, because when we use $R_{uv}$, it only changes the connected component which contains $uv$. Let's think about $P_{1}$ in the following:

                To proof `Formula (1)' is commutative, we should only proof any two consecutive operations is commutative. And if the edge of at least one operation in two consecutive operations is not in $E(G)$, they are commutative. So we should only focus on the case that the edges of two consecutive ones are both in $E(G)$. There are $3$ situations:

                $S_{1}:$ $R_{uv}$ and $R_{vw}$;

                $S_{2}:$ $R_{uv}$ and $R_{wx}$;

                $S_{3}:$ $R_{uv}$ and $R_{uv}$ (in fact, should be $\{u,v\}_{1}$ and $\{u,v\}_{2}$).

                Obviously we have $S_{3}$, because they are same when we don't care their names. It's not obvious for $S_{2}$, but we will not proof it here, because it's similar to and simpler than $S_{1}$ which we will proof in the following. The proof is simple but long, so, for clear, we use a matrix just to present the sum of its elements.
                \small
                \begin{align}
                &R_{vw}R_{uv}F(G)\notag&\\
                =&R_{vw}[(aF(G-uv)+bF(G\ /\ uv)+cF(G-u)+cF(G-v)+dF(G-u-v)]\notag\\
                =&\left(
                \begin{array}{lllll}
                \begin{smallmatrix}
                aaF(G-uv-vw) & abF(G-uv/vw) & acF(G-uv-v) & acF(G-uv-w) & adF(G-uv-v-w)\\
                abF(G/uv-vw) & bbF(G/uv/vw) & bcF(G/uv-v) & bcF(G/uv-w) & bdF(G/uv-v-w)\\
                acF(G-u-vw) & bcF(G-u/vw) & ccF(G-u-v) & ccF(G-u-w) & cdF(G-u-v-w)\\
                cF(G-v) &\\
                dF(G-u-v) &\\
                \end{smallmatrix}
                \end{array}
                \right).&
                \end{align}

				\normalsize
                We will change the order of some operations of graphs, such as $G-uv\ /\ vw=G\ /\ vw-uv$ and etc. They are all simple and we will not proof them here. So we have `Formula (2)'
                \small
                \begin{align}
                =&\left(
                \begin{array}{llll}
                \begin{smallmatrix}
                aaF(G-vw-uv) & abF(G/vw-uv) & acF(G-vw-v) & acF(G-w-uv) & adF(G-v-w)\\
                abF(G-vw/uv) & bbF(G/vw/uv) & bcF(G-u-v) & bcF(G-w/uv) & bdF(G/vw-u-w)\\
                acF(G-vw-u) & bcF(G/vw-u) & ccF(G-u-v) & ccF(G-u-w) & cdF(G-u-v-w)\\
                cF(G-v) &\\
                dF(G-u-v) &\\
                \end{smallmatrix}
                \end{array}
                \right).&
                \end{align}

                \normalsize
                In other side, if we change $u$ and $w$ in `$R_{vw}R_{uv}F(G)=$ Formula (2)', we will get
                \small
                \begin{align}
                &R_{uv}R_{vw}F(G)&\notag\\
                =&\left(
                \begin{array}{llll}
                \begin{smallmatrix}
                aaF(G-vw-uv) & abF(G-vw/uv) & acF(G-vw-v) & acF(G-vw-u) & adF(G-u-v)\\
                abF(G/vw-uv) & bbF(G/vw/uv) & bcF(G-v-w) & bcF(G/vw-u) & bdF(G/vw-u-w)\\
                acF(G-w-uv) & bcF(G-w/uv) & ccF(G-v-w) & ccF(G-u-w) & cdF(G-u-v-w)\\
                cF(G-v) &\\
                dF(G-v-w) &\\
                \end{smallmatrix}
                \end{array}
                \right).&
                \end{align}

                \normalsize
                Because of $R_{vw}R_{uv}F(G)=R_{uv}R_{vw}F(G)$, `Formula (3) $-$ Formula (4)' should equal to $0$. So we get $(bc+cc+d-ad)[F(G-u-v)-F(G-v-w)]=0$. So $F(\cdot)$ is well-defined if and only if $bc+cc+d-ad=0$ or $[F(G-u-v)-F(G-v-w)]=0$. Because of the Lemma 3, $F(\cdot)$ is well-defined if and only if $bc+cc+d-ad=0$ or $F(\cdot)$ is trivial.
            \end{proof}

            Because of the Theorem 4, $D(G)$ is  well-defined.
\section{An generaliztion of Tutte polynomial}
        When we talk about polynomial invariants of graphs, we must mention the Tutte polynomial $T(G,x,y)$ (the definition and properties of which is in \cite{dxs}). Next, we will merge $D(G)$ and $T(G)$:
        \begin{definition}
            For a graph $G$, $J(G,x,y,\lambda{},r,s,t)$ (sometimes short for $J(G)$) is got from:

            $(1)J(G)=J(G-uv)+tJ(G\ /\ uv)-\lambda{}tJ(G-u)-\lambda{}tJ(G-v)+rJ(G-u-v)$, $\lambda{}^{2}=\lambda{}$, $u\neq{}v;$

            $(2)J(G_{1}+G_{2})=J(G_{1})J(G_{2});$

            $(3)J(K_{1}^{n})=xy^{n}+nst$; $J(\phi{})=1.$
        \end{definition}

            Because of Theorem 4, $J(G)$ is well-defined. And then, we will give some properties of $J(G)$.
        \begin{proposition}
            For a graph $G:$

            $(1)J(G,0,1,1,0,s,1)=H(G,s);$

            $(2)J(G,1,1,1,t,s,t)=C(G,s,t);$

            $(3)J(G,1,1,1,r,s,t)=D(G,r,s,t);$

            $(4)J(G,x,y,0,0,0,1)=x^{k(G)}T(G,1+x,y);$

            $(5)[s^{k}]J(G,0,1,0,0,s,1)=\#\{H\subseteq\subseteq{}G\ |\ k(H)=k;H=\sum_{i=1}^{k}H_{i},\#E(H_{i}\\)=\#V(H_{i})\};$

            $(6)[s^{k}t^{l}]J(G,1,1,0,st-t,s,t)=\#\{H\subseteq{}G\ |\ k(H)=k;\#E(H)=l;H=\sum_{i=1}^{k}H_{i},\#E(H_{i})\leq{}\#V(H_{i})\}.$
        \end{proposition}
            \begin{proof}
                (1)-(3) is obvious from the foregoing. We can proof (4) by rewriting the definition of $x^{k(G)}T(G,1+x,y)$ so that it's same to the definition of $J(G,x,y,0,0,0,1)$, by using $T(G-uv,x,y)=T(G\ /\ uv,x,y)$ when $uv$ is a bridge in \cite{dxs}. And we will not proof (5) and (6) at here, because they are similar to the proof of Theorem 1.
            \end{proof}

            The subgraph in (5) is called a $k$-components spanning functional subgraph of $G$. A functional graph, which is defined by Frank Harary in \cite{fg}, is a graph that, in every component of which, there are same number of edges and vertices. Let's think a question: ``There are $n$ people, every two people in whom know or don't know each other. Everyone chooses an other person he knows. How many situations are there, which there are no two people choose each other?'' In fact, there are $\sum_{k}2^{k}[s^{k}]J(G,0,1,0,0,s,1)$ situations, i.e., $J(G,0,1,0,0,2,1)$, where $G$ is the relation graph of the $n$ people.

            Similarly, the subgraph in (6) is called a $k$-components $l$-length functional subgraph. When $l<\#V(H)$, it's not really a functional graph. In fact, every component is a functional graph or a tree. This is very interesting if we notice that every component is a cycle or a path in the study of the cover polynomial in \cite{cp}.
\section{A note}
            Finally, we will have more discussion on polynomial invariants of graphs. The meaning of recursion on edges is, when we focus on an edge, we can do something only on this edge, i.e., we can only see the part in the dotted line of the graph in Figure 1 when we focus on $uv$. So the Theorem 4 maybe have given all polynomial invariants of graph defined by recursion on edges.

            But, $F(G)$ is not enough to distinguish every pair of different graphs. For example, $F(G)$ can't distinguish the one in Figure 2 (which is got from an example in \cite{shou}).
            \begin{figure}[h]
            \centering
            \begin{minipage}[t]{0.4\linewidth}
            \centering
            \includegraphics[width=0.8\textwidth]{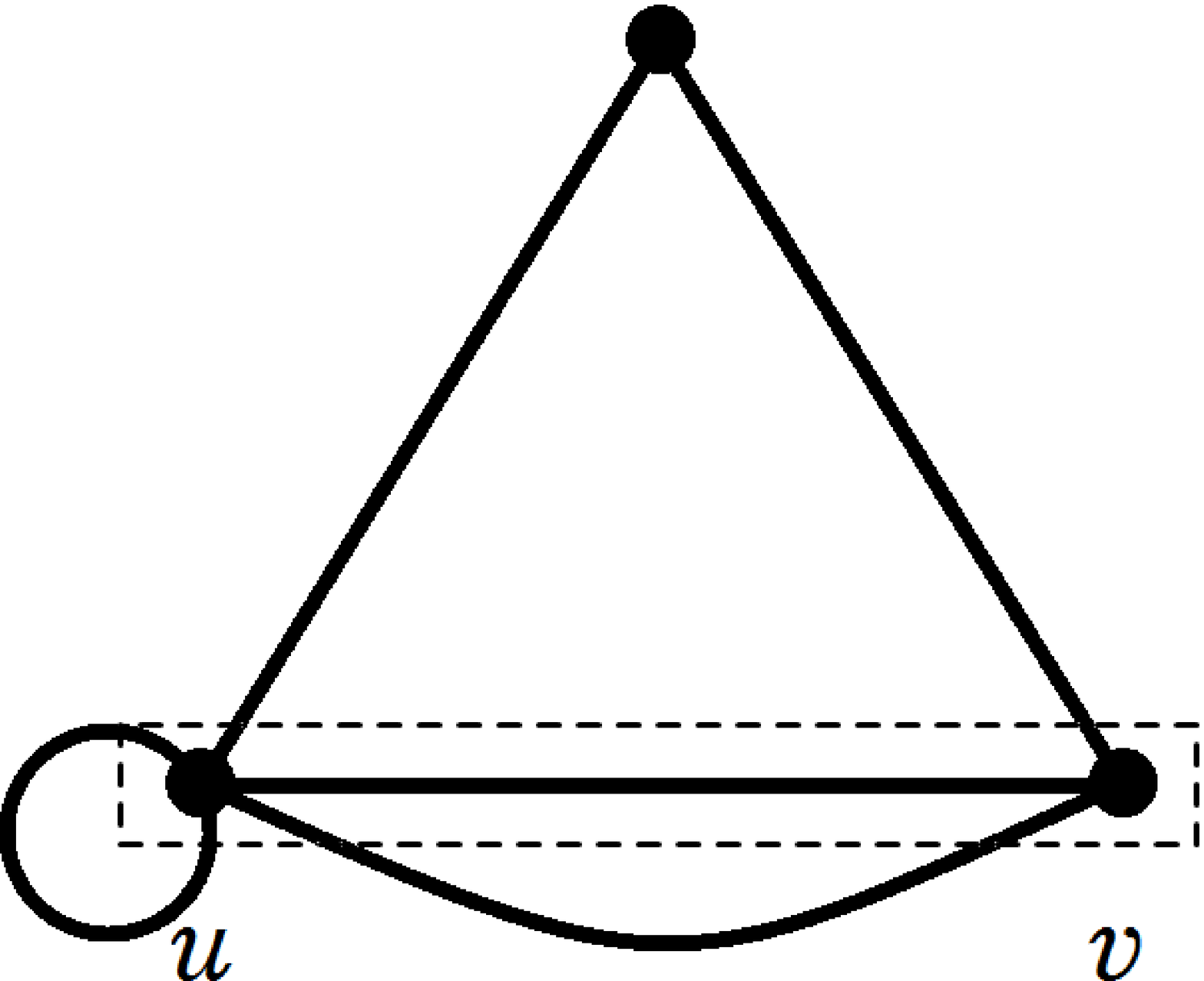}
            \centering
            \caption{}
            \label{uv}
            \end{minipage}
            \begin{minipage}[t]{0.5\linewidth}
            \centering
            \includegraphics[width=0.8\textwidth]{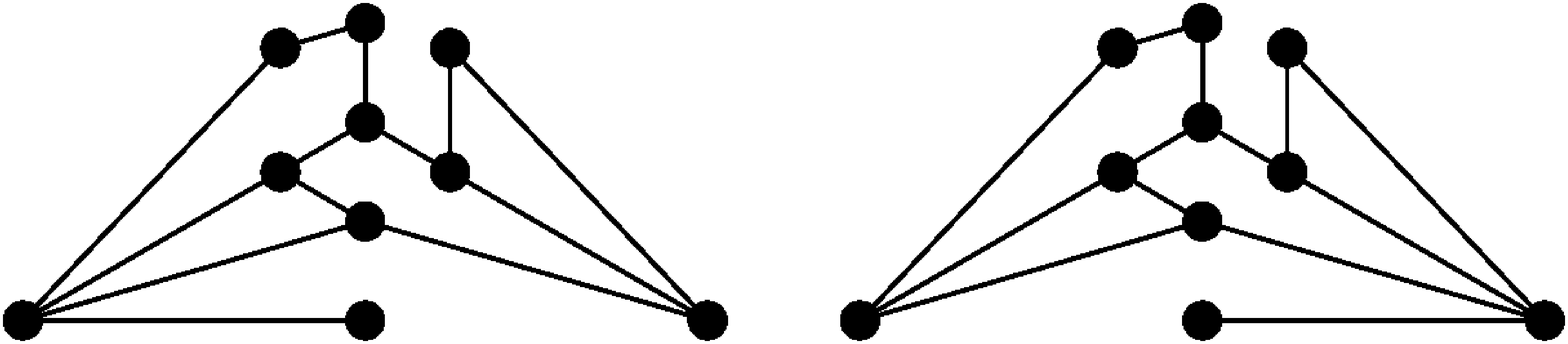}
            \centering
            \caption{}
            \label{shou}
            \end{minipage}
            \end{figure}

            This seems to be a big weak point of $F(G)$. But please remember, we only want to find a polynomial invariant which contain more invariants, especially the number of Hamiltonian cycles.

            But, in other side, we will give a plan, if we want to use $F(G)$ to distinguish graphs.

            $F(G)$ can be equivalently defined on the adjacency matrix $A(G)$ of $G$, which is denoted by $\hat{F}(A(G))$. Therefore, we can have $\hat{F}(Q(G))$ formally, where $Q(G)$ is the signless Laplacian matrix of $G$. Similarly we can have $\hat{C}(\cdot)$ and $\hat{J}(\cdot)$. The example above give a pair of graphs, which have a same $\hat{F}(A(G))$ but different $\hat{C}(Q(G))$.

            Moreover, Seiya Negami and Katsuhiro Ota give $9$ pairs of trees in \cite{tree}, each of which have a same $\hat{F}(A(G))$ and a same $\hat{C}(Q(G))$. But each of them have different $\hat{J}(Q(G))$.

            We hope, for every pair of different graphs, they have different $\hat{F}(Q(G))$. But, pessimistically, we guess, there will be counterexamples in strongly regular graph.

\begin{acknowledgements}
Thank Prof. Sheng Chen and Prof. Xunbo Yin for their help on this article.
\end{acknowledgements}



\end{document}